\documentclass{article}
\usepackage{amsmath}
\usepackage{amsfonts}

\newtheorem{theorem}{Theorem}

\newtheorem{corollary}[theorem]{Corollary}

\newtheorem{definition}[theorem]{Definition}

\newtheorem{lemma}[theorem]{Lemma}

\newtheorem{proposition}[theorem]{Proposition}

\newenvironment{proof}[1][Proof]{\noindent\textbf{#1.} }{\ \rule{0.5em}{0.5em}}

\begin{document}

\begin{center}
\textbf{On the real projection of the zeros of }$1+2^{z}+...+n^{z}$

\textit{Department of Mathematical Analysis, University of Alicante,
03080-Alicante, Spain.}

\textbf{E. Dubon, G. Mora, J.M. Sepulcre, J.I. Ubeda and T.
Vidal}
\end{center}

\textit{E-mail addresses: ed18@alu.ua.es, gaspar.mora@ua.es, JM.Sepulcre@ua.es, jiug@alu.ua.es, tmvg@alu.ua.es}

\textbf{ABSTRACT} This paper proves that the real projections of
the simple zeros of each partial sum $G_{n}(s)\equiv
1+2^{s}+...+n^{s}$, $n\geq 2$, of the Riemann zeta function on the
half-plane $\mbox{Re}s<-1,$ are not isolated points of that set.



\section{Introduction}

On zeros of partial sums of the Riemann zeta function we find several works.
Among them might be quoted the classical papers of Turan $\left[ 16\right] $
and Montgomery $\left[ 10\right] $, and more recently Borwein et al. $\left[
4\right] $, and Gonek and Ledoan $\left[ 6\right] $.

In this paper we focus on the existence of accumulation points of the subset
defined by the real projection of the zeros of the functions
\begin{equation*}
U_{n}(s):=\sum_{k=1}^{n}\frac{1}{k^{s}},\ s=\sigma +it.
\end{equation*}%
That would imply the existence of an infinite amount of zeros of $U_{n}(s)$
arbitrarily close to a line parallel to the imaginary axis passing through
every accumulation point. In order to study this problem, we begin by
considering the set defined by the real projection of the zeros of the
functions
\begin{equation*}
G_{n}(s):=1+2^{s}+...+n^{s}
\end{equation*}%
and then, by virtue of the equality%
\begin{equation*}
U_{n}(s)=G_{n}(-s)\mbox{ for all }s\in
%TCIMACRO{\U{2102} }%
%BeginExpansion
\mathbb{C}
%EndExpansion
,
\end{equation*}%
we have that the set of the zeros of $U_{n}(s)$, called $Z_{U_{n}(s)}$, is
the same as $-Z_{G_{n}(s)}$ where $Z_{G_{n}(s)}$ denotes the set of the
zeros of $G_{n}(s)$.

Each $G_{n}(s)$, $n\geq 2$, is an entire function of order $1$, of
exponential type $\sigma =\ln n$, and it has an infinite amount of zeros,
not all of them located on the imaginary axis, except for $G_{2}(s)$ $\left[
11\right] $. Furthermore, since for any $t$
\begin{equation*}
\lim_{\sigma \rightarrow -\infty }G_{n}(\sigma +it)=1
\end{equation*}%
and
\begin{equation*}
\lim_{\sigma \rightarrow +\infty }\frac{G_{n}(\sigma +it)}{n^{\sigma +it}}=1,
\end{equation*}%
there exist two values of $\sigma $, $\sigma _{n,1}<0<\sigma _{n,2}$, such
that
\begin{equation*}
\left\vert G_{n}(s)-1\right\vert <1\mbox{ for all }s\mbox{ with }\mbox{Re}%
s\leq \sigma _{n,1}
\end{equation*}%
and
\begin{equation*}
\left\vert \frac{G_{n}(\sigma +it)}{n^{\sigma +it}}-1\right\vert
<1\mbox{ for all }s\mbox{ with }\mbox{Re}s\geq \sigma
_{n,2}.
\end{equation*}

Thus, every function $G_{n}(s)$ has its zeros in a critical strip $S_{n}$
defined by
\begin{equation*}
S_{n}:=\left\{ s=\sigma +it:a_{n}\leq \sigma \leq b_{n}\right\},
\end{equation*}%
where the bounds
\begin{equation*}
a_{n}:=\inf \left\{ \mbox{Re}s:G_{n}(s)=0\right\}
\end{equation*}%
and%
\begin{equation*}
b_{n}:=\sup \left\{ \mbox{Re}s:G_{n}(s)=0\right\}
\end{equation*}%
are given by
\begin{equation*}
a_{n}=-1-\left( \frac{4}{\pi }-1+o(1)\right) \frac{\log \log n}{\log n}
\end{equation*}%
and
\begin{equation*}
b_{n}=n\log 2+o(1)
\end{equation*}%
as it was proved by Montgomery $\left[ 10\right] $ and Balazard and Vel%
\'{a}squez-Casta\~{n}\'{o}n $\left[ 3\right] $, respectively.

Since all zeros of the function $G_{2}(s)=1+2^{s}$ are on the imaginary
axis, the set defined by their real projections is reduced to $\left\{
0\right\} $, so $a_{2}=b_{2}=0$ and therefore we consider $n=2$ as the
trivial case.

In the case of $n>2$, we can use the bounds $a_{n}$, $b_{n}$ to define the
critical interval
\begin{equation*}
I_{n}:=\left[ a_{n},b_{n}\right]
\end{equation*}%
associated with each function\ $G_{n}(s)$ and that contains the set of the
real projections of the zeros of $G_{n}(s)$. The study on the existence of
accumulation points within that set is the main goal of the present paper.
In that sense, we will answer as well to several questions about this
subject that were pointed out in other publications like $\left[ 8\right] $,
$\left[ 9\right] $ and $\left[ 14\right] $.

Finally, we would like to note that, according to the author in $\left[ 13%
\right] $, the real projections of the zeros of exponential polynomials with
the form
\begin{equation*}
\varphi (s)=\sum_{k=1}^{m}A_{k}e^{\alpha _{k}s}
\end{equation*}%
where $\alpha _{k}$ are real exponents called frequencies, and under rather
restrictive conditions (refer to $\left[ 13\right] $) satisfies the
following so-called \textit{Main Theorem} that we would like to quote here:

MAIN THEOREM (Quoted from $\left[ 13\right] $). \textit{Assume that }$1$, $%
\alpha _{1}$,..., $\alpha _{m}$ \textit{are real numbers linearly
independent over the rationals. Consider the exponential polynomial }%
\begin{equation*}
\varphi (s)=\sum_{k=1}^{m}A_{k}e^{\alpha _{k}s},\ s=\sigma +it,
\end{equation*}%
\textit{where the }$A_{k}$ \textit{are complex numbers. Then a necessary and
sufficient condition for }$\varphi (s)$\textit{\ to have zeros arbitrarily
close to any line parallel to the imaginary axis inside the strip }%
\begin{equation*}
I=\left\{ \sigma +it:\sigma _{0}<\sigma <\sigma _{1},\ -\infty
<t<\infty \right\}
\end{equation*}%
\textit{\ is that }%
\begin{equation}
\left\vert A_{j}e^{\sigma \alpha _{j}}\right\vert \leq \sum_{k=1,\ %
k\neq j}^{m}\left\vert A_{k}e^{\sigma \alpha _{k}}\right\vert ,\ %
\left( j=1,2,...,m\right)   \tag{1.1}
\end{equation}%
\textit{for any }$\sigma $\textit{\ with }$\sigma +it\in I$.

From the above result it follows that the real projections of the zeros of\
an exponential polynomial of the form $\varphi (s)$, satisfying the
hypothesis of this theorem, define a dense subset of the real interval $%
(\sigma _{0},\sigma _{1})$ if and only if the geometric principle (1.1)
holds. In the case of the $G_{n}(s)$ functions, since only $G_{2}(s)$ and $%
G_{3}(s)$ are of the type $\varphi (s)$, in the sense that both satisfy the
condition of the linear independence over the rationals of their
frequencies, the \textit{Main Theorem} can be exclusively applied for $n=2$
and $n=3$. Nevertheless, it does not show that the projection on the real
axis of the zeros of $G_{n}(s)$, for $n>3$, might also have some density
properties in their critical intervals $\left[ a_{n},b_{n}\right] $, as it
happens, for instance, for the cases $n=4,5$ and other $n$'s, as we will see
in this paper.

\section{A characterization of the density of the real projections of the
zeros of $G_{n}(s)$}

As we said in the foregoing section it is crucial to study the existence of
accumulation points of the set of the real projection of the zeros of $%
G_{n}(s)$ within its critical interval $I_{n}$. Therefore, firstly we are
going to formulate a general theorem of characterization of the set
\begin{equation*}
R_{n}:=\overline{\left\{ \mbox{Re}s:G_{n}(s)=0\right\} }
\end{equation*}%
which can be considered as an ad hoc version of $\left[ 2\mbox{, Theorem 3.1}%
\right] $ and that can be directly applied to our functions $G_{n}(s).$

\begin{theorem}
For each integer $n>2$, let $\left\{ p_{1},p_{2},...,%
p_{k_{n}}\right\} $ be the set of all prime numbers less than or equal to $n$%
. Let us define $\mathbf{p}=(\log p_{1}$,$\log p_{2}$, ..., $\log
p_{k_{n}})$ and let $\mathbf{c}_{m}$ be the unique vector of $%
%TCIMACRO{\U{211d} }%
%BeginExpansion
\mathbb{R}
%EndExpansion
^{k_{n}}$with non-negative integer components such that $\log m=\left\langle
\mathbf{c}_{m}\mathbf{,p}\right\rangle $, $1\leq m\leq n$, where $%
\left\langle \ ,\ \right\rangle $ is the standard inner product in $%
%TCIMACRO{\U{211d} }%
%BeginExpansion
\mathbb{R}
%EndExpansion
^{k_{n}}$. Let us define the function $F_{n}:%
%TCIMACRO{\U{211d} }%
%BeginExpansion
\mathbb{R}
%EndExpansion
\times
%TCIMACRO{\U{211d} }%
%BeginExpansion
\mathbb{R}
%EndExpansion
^{k_{n}}\rightarrow
%TCIMACRO{\U{2102} }%
%BeginExpansion
\mathbb{C}
%EndExpansion
$ as
\begin{equation}
F_{n}(\sigma ,\mathbf{x}):=\sum_{m=1}^{n}m^{\sigma }e^{\left\langle \mathbf{c%
}_{m}\mathbf{,x}\right\rangle i}  \tag{2.1}
\end{equation}%
for $\sigma $ real and $\mathbf{x}=(x_{1}$, $x_{2}$, ... , $x_{k_{n}})$ a
vector of $%
%TCIMACRO{\U{211d} }%
%BeginExpansion
\mathbb{R}
%EndExpansion
^{k_{n}}$. Then,
\begin{equation*}
\sigma \in R_{n}:=\overline{\left\{ \mbox{Re}s:G_{n}(s)=0\right\}
}
\end{equation*}%
if and only if there exists some vector $\mathbf{x}\in
%TCIMACRO{\U{211d} }%
%BeginExpansion
\mathbb{R}
%EndExpansion
^{k_{n}}$ such that $F_{n}(\sigma ,\mathbf{x})=0$.
\end{theorem}

\begin{proof}
Firstly observe that if $s=\sigma +it$ is an arbitrary zero of $G_{n}(s)$,
since $\log m=\left\langle \mathbf{c}_{m},\mathbf{p}\right\rangle $, we can
write
\begin{equation*}
0=G_{n}(s)=\sum_{m=1}^{n}e^{\left\langle \mathbf{c}_{m}\mathbf{,p}%
\right\rangle s}=\sum_{m=1}^{n}e^{\left\langle \mathbf{c}_{m}\mathbf{,p}%
\right\rangle \sigma }e^{\left\langle \mathbf{c}_{m},\mathbf{p}\right\rangle
ti}=
\end{equation*}%
\begin{equation}
=\sum_{m=1}^{n}m^{\sigma }e^{\left\langle \mathbf{c}_{m},t\mathbf{p}%
\right\rangle i}=F_{n}(\sigma ,t\mathbf{p}).  \tag{2.2}
\end{equation}%
Now assume $\sigma \in R_{n}$. Then there exists a sequence $(s_{j}=\sigma
_{j}+it_{j})_{j=1,2,...}$of zeros of $G_{n}(s)$ such that $\sigma
=\lim_{j\rightarrow \infty }\sigma _{j}$ and, because of (2.2), we have
\begin{equation*}
F_{n}(\sigma _{j},t_{j}\mathbf{p})=0\mbox{, for all }j=1,2,...
\end{equation*}%
Hence, from (2.1), we obtain
\begin{equation}
0=\sum_{m=1}^{n}m^{\sigma _{j}}e^{\left\langle \mathbf{c}_{m},t_{j}\mathbf{p}%
\right\rangle i}\mbox{, for all }j=1,2,...  \tag{2.3}
\end{equation}%
Now consider the sequence $\left( e^{\left\langle \mathbf{c}_{2},t_{j}%
\mathbf{p}\right\rangle i}\right) _{j=1,2,...}$ of points of the unit
circle, thus bounded, and let
\begin{equation*}
\left( e^{\left\langle \mathbf{c}_{2},t_{j_{h,2}}\mathbf{p}\right\rangle
i}\right) _{h=1,2,...}
\end{equation*}%
be a convergent subsequence to $e^{\theta _{2}i}$ for some $\theta _{2}\in %
\left[ 0,2\pi \right) $. Since the sequence $\left( e^{\left\langle \mathbf{c%
}_{3},t_{j_{h,2}}\mathbf{p}\right\rangle i}\right) _{h=1,2,...}$is also
bounded, there exists a convergent subsequence
\begin{equation*}
\left( e^{\left\langle \mathbf{c}_{3},t_{j_{h,3}}\mathbf{p}\right\rangle
i}\right) _{h=1,2,...}
\end{equation*}%
to $e^{\theta _{3}i}$ for some $\theta _{3}\in \left[ 0,2\pi \right) $ and
so on. Then, by means of this process, for every prime number $m$ of the set
$\left\{ 1,2,3,...,n\right\} $, we determine a subsequence
\begin{equation*}
\left( e^{\left\langle \mathbf{c}_{m},t_{j_{h,p_{k_{n}}}}\mathbf{p}%
\right\rangle i}\right) _{h=1,2,...}
\end{equation*}%
which converges to $e^{\theta _{m}i}$ for some $\theta _{m}\in \left[ 0,2\pi
\right) $. For the other numbers of the set $\left\{ 1,2,3,...,n\right\} $
one has
\begin{equation*}
\mathbf{c}_{1}=(0,...,0),\mathbf{c}_{4}=2\mathbf{c}_{2},\mathbf{c}_{6}=%
\mathbf{c}_{2}+\mathbf{c}_{3},...,
\end{equation*}%
so, by considering (2.3) for the $j$'s of the subsequence $j_{h,p_{k_{n}}}$,
$h=1,2,...$, and taking the limit as $h\rightarrow \infty $ we get
\begin{equation*}
0=\lim_{h\rightarrow \infty }\sum_{m=1}^{n}m^{\sigma
_{j_{h,}p_{k_{n}}}}e^{\left\langle \mathbf{c}_{m},t_{j_{_{h,}p_{k_{n}}}}%
\mathbf{p}\right\rangle i}=
\end{equation*}%
\begin{equation*}
=1+2^{\sigma }e^{\theta _{2}i}+3^{\sigma }e^{\theta _{3}i}+4^{\sigma
}e^{2\theta _{2}i}+5^{\sigma }e^{\theta _{4}i}+6^{\sigma }e^{(\theta
_{2}+\theta _{3})i}+...=
\end{equation*}%
\begin{equation*}
=1+2^{\sigma }e^{\left\langle \mathbf{c}_{2},\mathbf{\theta }\right\rangle
i}+3^{\sigma }e^{\left\langle \mathbf{c}_{3},\mathbf{\theta }\right\rangle
i}+4^{\sigma }e^{\left\langle \mathbf{c}_{4},\mathbf{\theta }\right\rangle
i}+...+n^{\sigma }e^{\left\langle \mathbf{c}_{n},\mathbf{\theta }%
\right\rangle i}=F_{n}(\sigma ,\mathbf{\theta }),
\end{equation*}%
where $\mathbf{\theta }=(\theta _{2},\theta _{3},\theta _{5},...,\theta
_{p_{k_{n}}})$.

Conversely, suppose
\begin{equation*}
F_{n}(\sigma ,\mathbf{x})=0
\end{equation*}%
for some real number $\sigma $ and a vector $\mathbf{x}=(x_{1},$ $%
x_{2},...,x_{k_{n}})$ of $%
%TCIMACRO{\U{211d} }%
%BeginExpansion
\mathbb{R}
%EndExpansion
^{k_{n}}$. As the components of $\frac{1}{2\pi }\mathbf{p}=\left( \frac{1}{%
2\pi }\log p_{1},...,\frac{1}{2\pi }\log p_{k_{n}}\right) $ are linearly
independent over the rationals, given the numbers $\frac{1}{2\pi }%
\left\langle \mathbf{c}_{p_{1}},\mathbf{x}\right\rangle $, ... , $\frac{1}{%
2\pi }\left\langle \mathbf{c}_{p_{k_{n}}},\mathbf{x}\right\rangle $, $T=1$
and $\frac{\epsilon }{2^{j}2\pi }$, for each $j=1,2,...$ , by applying the
Kronecker theorem $\left[ 7\mbox{, p. 382}\right] $ there exist a sequence $%
\left( T_{j}\right) _{j=1,2,...}$, $T_{j}>1$, and integers $\left(
N_{j,l}\right) _{l=1,2,...,k_{n}}$ such that for each $j=1,2,...$ it holds
that
\begin{equation}
\left\vert T_{j}\frac{1}{2\pi }\log p_{l}-\frac{1}{2\pi }\left\langle
\mathbf{c}_{p_{l}}\mathbf{,x}\right\rangle -N_{j,l}\right\vert <\frac{%
\epsilon }{2^{j}2\pi }\mbox{ for all }l=1,2,...,k_{n}.  \tag{2.4}
\end{equation}%
Multiplying by $2\pi $ and substituting $\log p_{l}$ by $\left\langle
\mathbf{c}_{p_{l}}\mathbf{,p}\right\rangle $, the inequality (2.4) becomes
\begin{equation*}
\left\vert \left\langle \mathbf{c}_{p_{l}},T_{j}\mathbf{p}-\mathbf{x}%
\right\rangle -2\pi N_{j,l}\right\vert <\frac{\epsilon }{2^{j}}\mbox{ for
all }l=1,2,...,k_{n},
\end{equation*}%
which means that
\begin{equation*}
\lim_{j\rightarrow \infty }e^{\left\langle \mathbf{c}_{p_{l}},T_{j}\mathbf{p}%
-\boldsymbol{x}\right\rangle i}=1
\end{equation*}%
and then
\begin{equation}
\lim_{j\rightarrow \infty }e^{\left\langle \mathbf{c}_{p_{l}},\mathbf{x}%
-T_{j}\mathbf{p}\right\rangle i}=1. \tag{2.5}
\end{equation}%
Now, for each $m\in \left\{ 1,2,3,...,n\right\} $ and noticing that the
vector $\mathbf{c}_{m}$ is a linear combination with non-negative integer
coefficients of the vectors $\mathbf{c}_{p_{l}}$, $l=1,2,...,k_{n}$, it can
be deduced from (2.5) that
\begin{equation}
\lim_{j\rightarrow \infty }e^{\left\langle \mathbf{c}_{m},\mathbf{x}-T_{j}%
\mathbf{p}\right\rangle i}=1. \tag{2.6}
\end{equation}%
Thus, by using (2.6), we have
\begin{equation}
\lim_{j\rightarrow \infty }G_{n}(\sigma +iT_{j})=\lim_{j\rightarrow \infty
}\sum_{m=1}^{n}e^{\left\langle \mathbf{c}_{m}\mathbf{,p}\right\rangle \sigma
}e^{\left\langle \mathbf{c}_{m}\mathbf{,p}\right\rangle
iT_{j}}e^{\left\langle \mathbf{c}_{m},\mathbf{x}-T_{j}\mathbf{p}%
\right\rangle i}=  \notag
\end{equation}%
\begin{equation*}
=\lim_{j\rightarrow \infty }\sum_{m=1}^{n}e^{\left\langle \mathbf{c}_{m}%
\mathbf{,p}\right\rangle \sigma }e^{\left\langle \mathbf{c}_{m}\mathbf{,x}%
\right\rangle i}=\sum_{m=1}^{n}m^{\sigma }e^{\left\langle \mathbf{c}_{m}%
\mathbf{,x}\right\rangle i}=F_{n}(\sigma ,\mathbf{x})=0,
\end{equation*}%
which means that $\sigma \in R_{n}$ and then the result follows.
\end{proof}

By applying the preceding theorem to our function $G_{4}(s)$ we are going to
deduce the existence of a subinterval of its critical interval contained in $%
R_{4}$. This means that there is an infinite amount of zeros of $G_{4}(s)$
arbitrarily close to any line parallel to the imaginary axis passing through
any point of this subinterval.

\begin{corollary}
The interval $\left[ -0.55,1\right] $ is contained in $R_{4}:=\overline{%
\left\{ \mbox{Re}s:G_{4}(s)=0\right\} }$, where $%
G_{4}(s)=1+2^{s}+3^{s}+4^{s} $.
\end{corollary}

\begin{proof}
Because of (2.1), the function associated with $G_{4}(s)$ is given by
\begin{equation*}
F_{4}(\sigma ,x_{1},x_{2})=1+2^{\sigma }e^{x_{1}i}+3^{\sigma }e^{x_{2}i}+%
\mathit{\bigskip }4^{\sigma }e^{2x_{1}i}.
\end{equation*}%
By defining the function
\begin{equation*}
f_{4}(\sigma ,x_{1}):=1+2^{\sigma }e^{x_{1}i}+\mathit{\bigskip }4^{\sigma
}e^{2x_{1}i},
\end{equation*}%
we claim that for every $\sigma \in \left[ -0.55,1\right] $ there exists a
value of $x_{1}$ depending on $\sigma $, say $x_{1,\sigma }$, such that
\begin{equation*}
\left\vert f_{4}(\sigma ,x_{1,\sigma })\right\vert =3^{\sigma }.
\end{equation*}%
Indeed, the function $f_{4}(\sigma ,x_{1})$ can be rewritten as
\begin{equation*}
f_{4}(\sigma ,x_{1})=\left( 2^{\sigma }e^{x_{1}i}+e^{\frac{\pi }{3}i}\right)
\left( 2^{\sigma }e^{x_{1}i}+e^{-\frac{\pi }{3}i}\right),
\end{equation*}%
so its modulus is the product of the distances from the point $s=2^{\sigma
}e^{x_{1}i}$ to the fixed points $s_{0}=-e^{\frac{\pi }{3}i}$ and $%
s_{1}=-e^{-\frac{\pi }{3}i}$ of the unit circle. According to $\frac{1}{2}%
<2^{\sigma }$ for all $\sigma \in \left[ -0.55,1\right] $, the vertical line
of equation $x=-\frac{1}{2}$ intersects the circles of equation $\left\vert
s\right\vert =2^{\sigma }$ at points $P_{1}$, $P_{2}$ and $P_{3}$ of the
upper half-plane for each possible value of $\sigma $ in the interval $\left[
-0.55,1\right] $, namely, $\sigma <0$, $\sigma =0$ and $\sigma >0$,
respectively. Let us denote $Q_{1}$, $Q_{2}$, $Q_{3}$ as the points where
the positive real half-axis intersects the circles $\left\vert s\right\vert
=2^{\sigma }$ for $\sigma <0$, $\sigma =0$ and $\sigma >0$ respectively,
then we may have the following situations:

1) Case $-0.55\leq \sigma <0$. The modulus $\left\vert
f_{4}(P_{1})\right\vert =1-4^{\sigma }$ satisfies the inequality $%
1-4^{\sigma }\leq 3^{\sigma }$. On the other hand, $\left\vert
f_{4}(Q_{1})\right\vert =1+2^{\sigma }+4^{\sigma }\geq 3^{\sigma }$ for any $%
-0.55\leq \sigma <0$.

2) Case $\sigma =0$. We have $\left\vert f_{4}(P_{2})\right\vert =0$ and $%
\left\vert f_{4}(Q_{2})\right\vert =3$.

3) Case $0<\sigma \leq 1$. In this case $\left\vert f_{4}(P_{3})\right\vert
=4^{\sigma }-1\leq 3^{\sigma }$ for any $0<\sigma \leq 1$. On the other
hand, $\left\vert f_{4}(Q_{3})\right\vert =1+2^{\sigma }+4^{\sigma }\geq
3^{\sigma }$ for any $0<\sigma \leq 1$.

Then, taking into account the above three cases and from the
continuity of the function $\left\vert f_{4}(\sigma
,x_{1})\right\vert $, the mean value theorem implies that for each
$\sigma$ in $\left[ -0.55,1\right] $ there is a value $x_{1,\sigma
}$ such that
\begin{equation*}
\left\vert f_{4}(\sigma ,x_{1,\sigma })\right\vert =3^{\sigma },
\end{equation*}%
as we claimed. Thus $f_{4}(\sigma ,x_{1,\sigma })$ is a point of the circle $%
\left\vert s\right\vert =3^{\sigma }$, so there exists a real number, say $%
x_{2}$, such that
\begin{equation*}
f_{4}(\sigma ,x_{1,\sigma })=-3^{\sigma }e^{x_{2}i}.
\end{equation*}%
Now, taking into account that for any arbitrary real numbers $\sigma
,x_{1},x_{2}$ it is true that
\begin{equation*}
F_{4}(\sigma ,x_{1},x_{2})=f_{4}(\sigma ,x_{1})+3^{\sigma }e^{x_{2}i},
\end{equation*}%
it follows that $F_{4}(\sigma ,x_{1,\sigma },x_{2})=0$. Therefore the proof
is completed.
\end{proof}

In the next result we will prove the existence of a monotone relationship
between the sets $R_{n}$ and $R_{n+1}$ when $n+1$ is a prime number. But
before that, we will introduce a new real interval associated with each
function $G_{n}(s)$.

\begin{definition}
For each integer $n\geq $ $2$ we define the numbers
\begin{equation*}
x_{n,_{0}}:=\inf \left\{ \sigma \in
%TCIMACRO{\U{211d} }%
%BeginExpansion
\mathbb{R}
%EndExpansion
:1\leq 2^{\sigma }+...+n^{\sigma }\right\}
\end{equation*}%
and%
\begin{equation*}
x_{n,_{1}}:=\sup \left\{ \sigma \in
%TCIMACRO{\U{211d} }%
%BeginExpansion
\mathbb{R}
%EndExpansion
:1+2^{\sigma }+...+(n-1)^{\sigma }\geq n^{\sigma }\right\}.
\end{equation*}
\end{definition}

It is immediate to check that $x_{2,0}=x_{2,1}=0$, and for $n>2$ it holds
that $\left[ 0,1\right] \subset \left[ x_{n,_{0}},x_{n,_{1}}\right] $.

The relationship between the new interval $\left[ x_{n,_{0}},x_{n,_{1}}%
\right] $ and the critical interval $\left[ a_{n},b_{n}\right] $ is shown in
the following lemma.

\begin{lemma}
Let $I_{n}=\left[ a_{n},b_{n}\right] $ be the critical interval associated
with the function $G_{n}(s)$. Then $\left[ a_{n},b_{n}\right] \subset \left[
x_{n,0},x_{n,1}\right] $ for all $n\geq 2$.
\end{lemma}

\begin{proof}
Since $a_{2}=b_{2}=x_{2,0}=x_{2,1}=0$, the lemma trivially follows for $n=2$%
. Now, let us assume $n>2$. In that case, it can be seen that\ $b_{n}\leq
x_{n,1}$. Indeed, since
\begin{equation*}
b_{n}:=\sup \left\{ \mbox{Re}s:G_{n}(s)=0\right\},
\end{equation*}%
by assuming that $b_{n}>x_{n,1}$, there exists $w=a+ib$ a zero of
$G_{n}(s)$ such that $a>$ $x_{n,1}$. Thus, because of the
definition of $x_{n,1}$, it follows that
\begin{equation}
1+2^{a}+...+(n-1)^{a}<n^{a}. \tag{2.7}
\end{equation}%
On the other hand, as $G_{n}(w)=0$, we write
\begin{equation*}
1+2^{w}+...+(n-1)^{w}=-n^{w},
\end{equation*}%
and taking the modulus we obtain
\begin{equation*}
n^{a}\leq 1+2^{a}+...+(n-1)^{a},
\end{equation*}%
which contradicts (2.7). Hence $b_{n}\leq x_{n,1}$, as claimed before.

Likewise we are going to prove that $a_{n}\geq x_{n,0}$ by proof by
contradiction. Hence, let us assume that $a_{n}<x_{n,0}$. By noticing that%
\begin{equation*}
a_{n}:=\inf \left\{ \mbox{Re}s:G_{n}(s)=0\right\},
\end{equation*}%
there exists $u=c+id$ a zero of $G_{n}(s)$ with $c<$ $x_{n,0}$. Therefore,
from the definition of $x_{n,0}$, it follows that
\begin{equation}
1>2^{c}+...+n^{c}.  \tag{2.8}
\end{equation}%
Now, as $G_{n}(u)=0$, we get
\begin{equation*}
2^{u}+...+n^{u}=-1
\end{equation*}%
and taking the modulus we are led to
\begin{equation*}
1=\left\vert 2^{u}+...+n^{u}\right\vert \leq 2^{c}+...+n^{c},
\end{equation*}%
which contradicts (2.8). Therefore $a_{n}\geq x_{n,0}$,\ as claimed, and so
we have
\begin{equation*}
\left[ a_{n},b_{n}\right] \subset \left[ x_{n,0},x_{n,1}\right],
\end{equation*}%
which proves the lemma.
\end{proof}

\begin{proposition}
Let $n+1$ be a prime number greater than $2$, $I_{n}=\left[ a_{n},b_{n}%
\right] $ the critical interval of $G_{n}(s)$ and
$R_{n}:=\overline{\left\{ \mbox{Re}s:G_{n}(s)=0\right\} }$. Then
\begin{equation*}
R_{n}\cap \left[ a_{n},\frac{\log 2}{\log (1+\frac{1}{n})}\right] \subset
R_{n+1}.
\end{equation*}
\end{proposition}

\begin{proof}
Since $R_{2}=\left\{ 0\right\} $, $a_{2}=0$, and the fact that $0\in R_{3}$,
see $\left[ 13\right] $, the proposition follows for $n=2$. Therefore let us
assume that $n>2$. Let $\sigma $ be a point of $R_{n}\cap \left[ a_{n},%
\frac{\log 2}{\log (1+\frac{1}{n})}\right] $. Because of Theorem 1, there
exists a vector $\mathbf{x}\in
%TCIMACRO{\U{211d} }%
%BeginExpansion
\mathbb{R}
%EndExpansion
^{k_{n}}$ such that the function $F_{n}(\sigma ,\mathbf{x})=0$, so
\begin{equation}
\left\vert F_{n}(\sigma ,\mathbf{x})\right\vert =0.  \tag{2.9}
\end{equation}%
On the other hand, from Lemma 4,
\begin{equation*}
\sigma \in R_{n}\subset \left[ a_{n},b_{n}\right] \subset \left[
x_{n,0},x_{n,1}\right],
\end{equation*}%
and subsequently
\begin{equation*}
1+2^{\sigma }+...+(n-1)^{\sigma }\geq n^{\sigma }.
\end{equation*}%
Hence
\begin{equation}
1+2^{\sigma }+...+(n-1)^{\sigma }+n^{\sigma }\geq 2n^{\sigma }.
\tag{2.10}
\end{equation}%
Taking into account that
\begin{equation*}
\sigma \leq \frac{\log 2}{\log (1+\frac{1}{n})},
\end{equation*}%
we can write
\begin{equation*}
\left(1+\frac{1}{n}\right)^{\sigma }\leq 2
\end{equation*}%
and multiplying by $n^{\sigma }$ it can be deduced that
\begin{equation*}
(n+1)^{\sigma }\leq 2n^{\sigma }.
\end{equation*}%
Now, from (2.10), we get
\begin{equation*}
1+2^{\sigma }+...+n^{\sigma }\geq (n+1)^{\sigma }
\end{equation*}%
and then%
\begin{equation}
\left\vert F_{n}(\sigma ,\mathbf{0})\right\vert =1+2^{\sigma }+...+n^{\sigma
}\geq (n+1)^{\sigma },  \tag{2.11}
\end{equation}%
where $\mathbf{0}$ denotes the vector zero of $%
%TCIMACRO{\U{211d} }%
%BeginExpansion
\mathbb{R}
%EndExpansion
^{k_{n}}$. Furthermore, because of the continuity of the modulus of $F_{n}(\sigma ,%
\mathbf{x})$ and using (2.9) and (2.11), there exists a vector
$\mathbf{a}=(a_{1},...,a_{k_{n}})\in \mathbb{R}^{k_{n}}$ such that
\begin{equation*}
\left\vert F_{n}(\sigma ,\mathbf{a}))\right\vert =(n+1)^{\sigma },
\end{equation*}%
and then, for some $\alpha \in \left[ 0,2\pi \right) $, we write
\begin{equation}
F_{n}(\sigma ,\mathbf{a})=(n+1)^{\sigma }e^{\alpha i}.  \tag{2.12}
\end{equation}%
Since $n+1$ is a prime number, the number of prime numbers, $k_{n+1}$, of
the sequence $\left\{ 1,2,...,n+1\right\} $ is so that $k_{n+1}$ $=k_{n}+1$
and then, noticing (2.1), we put
\begin{equation}
F_{n+1}(\sigma ,\mathbf{y})=F_{n}(\sigma ,\mathbf{x}_{y})+(n+1)^{\sigma
}e^{y_{^{k_{n+1}}}i},  \tag{2.13}
\end{equation}%
where $\mathbf{y}$ is an arbitrary vector of $%
%TCIMACRO{\U{211d} }%
%BeginExpansion
\mathbb{R}
%EndExpansion
^{k_{n+1}}$, $\mathbf{x}_{y}$ is the vector of $%
%TCIMACRO{\U{211d} }%
%BeginExpansion
\mathbb{R}
%EndExpansion
^{k_{n}}$ defined by the first $k_{n}$ components of $\mathbf{y}$ and $%
y_{k_{n+1}}$ is the last component of $\mathbf{y}$. Thus, by substituting in
(2.13) the vector $\mathbf{y}$ by the vector $\mathbf{b:}%
=(a_{1},...,a_{k_{n}},\alpha +\pi )$ and, according to (2.12), it follows
\begin{equation*}
F_{n+1}(\sigma ,\mathbf{b})=0,
\end{equation*}%
which means, from Theorem 1, that $\sigma \in R_{n+1}$ . Now the proof is
completed and so the proposition follows.
\end{proof}

\begin{corollary}
The interval $\left[ -0.55,1\right] $ is contained in $R_{5}:=\overline{%
\left\{ \mbox{Re}s:G_{5}(s)=0\right\}}$, where $%
G_{5}(s)=1+2^{s}+3^{s}+4^{s}+5^{s}$.
\end{corollary}

\begin{proof}
From Corollary 2 we have $\left[ -0.55,1\right] \subset R_{4}\subset \left[
a_{4},b_{4}\right] $. On the other hand, as $5$ is a prime number, because
of the above proposition, we get
\begin{equation*}
R_{4}\cap \left[ a_{4},\frac{\log 2}{\log (1+\frac{1}{4})}\right] \subset
R_{5}.
\end{equation*}%
Now, noticing
\begin{equation*}
\left[ -0.55,1\right] \subset \left[ a_{4},\frac{\log 2}{\log (1+\frac{1}{4}%
)}\right],
\end{equation*}
the result follows.
\end{proof}

\section{Density properties of $G_{n}(s)$}

In this section we are going to study some properties of the functions $%
G_{n}(s)$ about the existence of zeros arbitrarily close to any right-line
parallel to the imaginary axis contained in its critical strip. Firstly,
taking into account the definition of $x_{n,0}$ and $x_{n,1}$, we observe
that in the strip
$$\left\{ s=\sigma +it:x_{n,0}<\sigma <x_{n,1}\right\},$$
the functions $G_{n}(s)$ satisfy the following geometric
principle.

\begin{lemma}
Let $n$ be an integer greater than $2$. Then, for arbitrary $s=\sigma +it$
inside the strip $\left\{ s\in
%TCIMACRO{\U{2102} }%
%BeginExpansion
\mathbb{C}
%EndExpansion
:x_{n,0}\leq \mbox{Re}s\leq x_{n,1}\right\} $, it follows that
\begin{equation}
\left\vert j^{s}\right\vert \leq \sum_{k=1;k\neq j}^{n}\left\vert
k^{s}\right\vert \mbox{, for every }j=1,2,...,n.  \tag{3.1}
\end{equation}
\end{lemma}

\begin{proof}
For $j=1$ and $j=n$, inequality (3.1) is immediate by virtue of the
definitions of $x_{n,0}$ and $x_{n,1}$, respectively. For any $j\neq 1,n$,
inequality (3.1) is true for arbitrary $s\in
%TCIMACRO{\U{2102} }%
%BeginExpansion
\mathbb{C}
%EndExpansion
$. Then the lemma follows.
\end{proof}

\begin{corollary}
For any $\sigma \in \left[ a_{n},b_{n}\right] $ there exists at least one $n$%
-sided polygon whose sides have lengths $1$, $2^{\sigma }$, ..., $n^{\sigma
} $.
\end{corollary}

\begin{proof}
Because of Lemma 4, we have $\left[ a_{n},b_{n}\right] \subset \left[
x_{n,0},x_{n,1}\right] $. Now, taking into account the preceding lemma, $%
\sigma $ is a real number such that the lengths $1,2^{\sigma }, ..., n^{\sigma}$ satisfy all the inequalities (3.1). Hence, from $\left[ 13\mbox{%
, p. 71}\right] $, the conclusion of the corollary is valid.
\end{proof}

\begin{lemma}
Every function $G_{n}(s):=1+2^{s}+...+n^{s}$, $n\geq $ $2$, satisfies the
following properties: \newline
(a) $G_{n}(s)$ is bounded on the strip $S_{n}:=\left\{ s=\sigma
+it:a_{n}\leq \sigma \leq b_{n}\right\} $. \newline
(b) There exists some $\sigma _{0}\in \left[ a_{n},b_{n}\right] $ such that
the vertical line of equation $x=\sigma _{0}$ contains a sequence of points $%
(\sigma _{0}+iT_{j})_{j=1,2,...}$ such that
\begin{equation*}
\lim_{j\rightarrow \infty }G_{n}(\sigma _{0}+iT_{j})=0.
\end{equation*}%
(c) There exist positive numbers $\delta $ and $l$ such that on any segment
of length $l$ of the line $x=\sigma _{0}$ there is a point $\sigma _{0}+iT$
such that $\left\vert G_{n}(\sigma _{0}+iT)\right\vert \geq \delta $.
\end{lemma}

\begin{proof}
The case $n=2$ is trivial. Indeed, $S_{2}$ coincides with the imaginary axis
so $\left\vert G_{2}(s)\right\vert \leq 2$ for all $s\in S_{2}$ and then (a)
holds. The line of equation $x=0$ contains all the zeros of $G_{2}(s)$,
namely, the imaginary numbers $s=\frac{i\pi (2k+1)}{\ln 2}$, $k\in
%TCIMACRO{\U{2124} }%
%BeginExpansion
\mathbb{Z}
%EndExpansion
$, so (b) is true. Moreover, by taking $\delta =2$ and an arbitrary $l>%
\frac{2\pi }{\ln 2}$, any segment of length $l$ of the\ right-line $x=0$
contains a point $iT$, with $T=\frac{i\pi 2k}{\ln 2}$ for some integer $k
$, such that $\left\vert G_{n}(iT)\right\vert \geq 2$. Hence, for $n=2$, the
result follows. Now let us assume that $n>2$.

It is clear that $G_{n}(s)$ is bounded on any finite vertical strip, so, in
particular, $G_{n}(s)$ is bounded on $S_{n}$. Therefore part (a) is proved.

Consider an arbitrary zero $s_{0}=\sigma _{0}+it_{0}$ of $G_{n}(s)$. By
applying Theorem 1, the vector
\begin{equation*}
t_{0}\mathbf{p=}t_{0}(\log p_{1},\log p_{2},...,\log p_{k_{n}})
\end{equation*}%
is such that $F_{n}(\sigma _{0},t_{0}\mathbf{p})=0$, where $F_{n}$ is the
function defined by (2.1). Now, by following the proof of the sufficiency of
the mentioned Theorem 1, there exists a sequence of points $(\sigma
_{0}+iT_{j})_{j=1,2,...}$ on the line $x=\sigma _{0}$ such that
\begin{equation*}
\lim_{j\rightarrow \infty }G_{n}(\sigma _{0}+iT_{j})=0.
\end{equation*}%
Consequently, part (b) has been shown.

Finally, since $G_{n}(\sigma _{0})>0,$ and taking into account that $G_{n}(s)
$ is an almost-periodic function, given $\delta =\frac{G_{n}(\sigma _{0})}{2%
}$, there exists a real number $l=l(\delta )$\ such that every
interval of length $l$\ on the imaginary axis contains at least one
translation number $T$,\ associated with $\delta $,\ satisfying the
inequality
\begin{equation*}
\left\vert G_{n}(s+iT)-G_{n}(s)\right\vert \leq \delta \mbox{, for all }s\in
%TCIMACRO{\U{2102} }%
%BeginExpansion
\mathbb{C}
%EndExpansion
.
\end{equation*}%
In particular, for $s=\sigma _{0}$, one has $\left\vert G_{n}(\sigma
_{0}+iT)-G_{n}(\sigma _{0})\right\vert \leq \delta $ and, according to the
choice of $\delta $, it implies that
\begin{equation*}
\left\vert G_{n}(\sigma _{0}+iT)\right\vert \geq \delta .
\end{equation*}%
Now, the proof of the lemma is completed.
\end{proof}

The density properties on the real projections of the zeros of $G_{n}(s)$
are given by means of the following result.

\begin{theorem}
Let $\sigma _{0}$ be a point of the critical interval $\left[ a_{n},b_{n}%
\right] $ of $G_{n}(s)$, $n\geq 2$, verifying properties (b) and (c) of the
above lemma. Then $G_{n}(s)$ has zeros in the strip
\begin{equation*}
\left\{ s\in
%TCIMACRO{\U{2102} }%
%BeginExpansion
\mathbb{C}
%EndExpansion
:\sigma _{0}-\delta \leq \mbox{Re}s\leq \sigma _{0}+\delta
\right\},
\end{equation*}%
for an arbitrary $\delta $ satisfying $0\leq \delta \leq \min \left\{ \sigma
_{0}-a_{n},b_{n}-\sigma _{0}\right\} $.
\end{theorem}

\begin{proof}
The function $G_{n}(s)$ and $\sigma _{0}$ satisfy the properties (a), (b)
and (c) of the previous lemma. Then, it suffices to apply the result
attributed in $\left[ 13\mbox{, p. 74}\right] $ to H. Bohr for demonstrating
the validity of the theorem.
\end{proof}

A more detailed result about the existence of zeros of the $G_{n}(s)$
functions in a strip is provided in Section 4 of this paper.

\section{A characterization of the density of the real projections of the
zeros of $G_{n}(s)$ by means of level curves}

In this section a new characterization of the sets $R_{n}:=\overline{\left\{
\mbox{Re}s:G_{n}(s)=0\right\} }$ in terms of the old concept of level curve $%
\left[ 15\right] $ is shown. Firstly we examine an important
property of the $G_{n}(s)$ functions.

\begin{lemma}
Given $\sigma \in
%TCIMACRO{\U{211d} }%
%BeginExpansion
\mathbb{R}
%EndExpansion
$, for each $n\geq 2$ the function $G_{n}(s):=1+2^{s}+...+n^{s}$ satisfies%
\begin{equation*}
Max\left\{ \left\vert G_{n}(s)\right\vert :\mbox{Re}s\leq \sigma
\right\} =G_{n}(\sigma ).
\end{equation*}%
Furthermore, for $n>2$, $s=\sigma $ is the unique point where the maximum is
attained.
\end{lemma}

\begin{proof}
Given $\sigma $, let $s=x+iy$ be any complex number with $x\leq \sigma $.
Since
\begin{equation*}
\left\vert G_{n}(s)\right\vert \leq G_{n}(x)\leq G_{n}(\sigma ),
\end{equation*}%
it can be deduced that
\begin{equation*}
Max\left\{ \left\vert G_{n}(s)\right\vert :\mbox{Re}s\leq \sigma
\right\} =G_{n}(\sigma ),
\end{equation*}%
so the first part of the lemma follows. Now let us assume that $n>2$, to
prove that $s=\sigma $ is the unique point where the maximum is attained it
suffices to show that
\begin{equation}
\left\vert 1+2^{\sigma +it}+3^{\sigma +it}\right\vert <1+2^{\sigma
}+3^{\sigma }\mbox{ for all real }t\neq 0.  \tag{4.1}
\end{equation}%
Indeed,
\begin{equation*}
\left\vert 1+2^{\sigma +it}+3^{\sigma +it}\right\vert =\left\vert \frac{1}{2}%
+2^{\sigma +it}+\frac{1}{2}+3^{\sigma +it}\right\vert \leq \left\vert \frac{1%
}{2}+2^{\sigma +it}\right\vert +\left\vert \frac{1}{2}+3^{\sigma
+it}\right\vert.
\end{equation*}%
Now we claim that at least one of the two inequalities
\begin{equation*}
\left\vert \frac{1}{2}+2^{\sigma +it}\right\vert \leq \frac{1}{2}+2^{\sigma }%
,\ \left\vert \frac{1}{2}+3^{\sigma +it}\right\vert \leq \frac{1}{2}%
+3^{\sigma }
\end{equation*}%
must be strict. Otherwise, there will be two positive $\lambda $ and $\mu $
such that
$$2^{\sigma +it}=\frac{\lambda}{2},\ \ 3^{\sigma +it}=\frac{\mu}{2}.$$ However there are integers $k$, $l\neq 0$ such that $\frac{\log 2}{%
\log 3}=\frac{k}{l}$ which contradicts the Fundamental Theorem of
Arithmetic. Hence (4.1) is true and the proof of the lemma is
completed.
\end{proof}

Following $\left[ 15\mbox{, p.121}\right] $ we recall the concept of level
curves.

\begin{definition}
Given an entire function $f(z)$ and a non-negative constant $k$, the curves
defined by the equation
\begin{equation}
\left\vert f(x+iy)\right\vert =k  \tag{4.2}
\end{equation}%
are called level curves of order $k$.
\end{definition}

For example, the level curves of the exponential function $e^{z}$ of order $%
k>0$ are the vertical lines of equations $x=\log k$. The level curves of $%
G_{2}(z):=1+2^{z}$, $z=x+iy$, are given by the equation
\begin{equation*}
1+2^{x+1}\cos (y\log 2)+2^{2x}=k^{2}\mbox{, for every }k>0,
\end{equation*}%
which does not contain any vertical line. Therefore, because of Lemma 11,
the level curves corresponding to the functions $G_{n}(z)$, $n\geq 2$, of
order $k>0$, do not contain any vertical line.

We are going to prove that if equation (4.2) has at least one solution, say $%
z_{0}=x_{0}+iy_{0}$, then there exists only one level curve in a certain
neighborhood of $(x_{0},y_{0})$ that passes through the point $(x_{0},y_{0})$%
, provided that $f^{\prime }(z_{0})\neq 0$.

\begin{lemma}
Let $f(z)$ be an entire function, $k>0$ and $z_{0}=x_{0}+iy_{0}$ a point
satisfying the equation $\left\vert f(z)\right\vert =$ $k$ . Then, if $%
f^{\prime }(z_{0})\neq 0$, in certain neighborhood of $(x_{0},y_{0})$ there
exists only one level curve, either defined by a function $\varphi $ or by a
function $\psi $, both of class $\mathcal{C}^{\infty }$, passing through the
point $(x_{0},y_{0})$ and such that $(\varphi (y),y)$ or $(x,\psi (x))$
satisfy (4.2) on certain neighborhoods of $y$ and $x$, respectively.
\end{lemma}

\begin{proof}
\ By setting $f=u+iv$, equation (4.2) can be written as $\Phi (x,y)=0$,
where $\Phi (x,y):=u^{2}+v^{2}-k^{2}$. Then, if we assume $\frac{\partial
\Phi }{\partial x}(x_{0},y_{0})=\frac{\partial \Phi }{\partial y}%
(x_{0},y_{0})=0$, it follows that
\begin{equation}
\left.
\begin{array}{c}
2u\frac{\partial u}{\partial x}+2v\frac{\partial v}{\partial x}=0 \\
2u\frac{\partial u}{\partial y}+2v\frac{\partial v}{\partial y}=0%
\end{array}%
\right\}.  \tag{4.3}
\end{equation}%
Now, by virtue of Cauchy-Riemann equations, the expression (4.3) becomes
\begin{equation}
\left.
\begin{array}{c}
u\frac{\partial u}{\partial x}+v\frac{\partial v}{\partial x}=0 \\
v\frac{\partial u}{\partial x}-u\frac{\partial v}{\partial x}=0%
\end{array}%
\right\}.  \tag{4.4}
\end{equation}%
Now, since the determinant of the matrix of system (4.4) is $%
-u^{2}-v^{2}=-k^{2}\neq 0$, the system (4.4) only has the solution $\frac{%
\partial u}{\partial x}(x_{0},y_{0})=\frac{\partial v}{\partial x}%
(x_{0},y_{0})=0$. It means that $f^{\prime }(z_{0})=0$, which is a
contradiction. Hence, either $\frac{\partial \Phi }{\partial x}%
(x_{0},y_{0})\neq 0$ or $\frac{\partial \Phi }{\partial y}(x_{0},y_{0})\neq
0$ and then, from the implicit function theorem, the result follows.
\end{proof}

The existence of a level curve passing through a non-critical point is
guaranteed by means of the following result.

\begin{proposition}
Let $f(z)$ be a non-constant entire function. For every $k>0$, there exists
a point $z_{0}$ with $f^{\prime }(z_{0})\neq 0$ belonging to the level curve
$\left\vert f(z)\right\vert =$ $k$.
\end{proposition}

\begin{proof}
According to Picard's little theorem $\left[ 5\mbox{, p. 432}\right] $, for
every $w$ of the circle $C_{k}:=\left\{ w:\left\vert w\right\vert =k\right\}
$ the equation $f(z)=w$ has a solution except, at most, for a value, say $%
w_{0}$, of $C_{k}$. Since the set $f^{-1}\left( C_{k}\setminus \left\{
w_{0}\right\} \right) $ is uncountable and $D_{f}:=\left\{ z:f^{\prime
}(z)=0\right\} $ is a countable set, let $z_{0}$ be a point of
\begin{equation*}
f^{-1}\left( C_{k}\setminus \left\{ w_{0}\right\} \right) \setminus D_{f}%
.
\end{equation*}%
Then, the point $z_{0}$ satisfies the equation $\left\vert f(z)\right\vert =k
$ and it is such that $f^{\prime }(z_{0})\neq 0$. Consequently, the
proposition follows.
\end{proof}

\begin{proposition}
Let $f(z)$ be analytic on an open set $U$. Let $V\subset U$ be a proper
bounded open subset such that $f^{\prime }(z)\neq 0$ for all $z\in V$.
Assume the level curve $\left\vert f(z)\right\vert =k$, $k>0$, passes
through a point $z_{0}\in V$. Then, either the level curve is a closed curve
in $\overline{V}$ or it has two disjoint arc-connected closed subsets $A$, $%
B $ of the boundary of $V$.
\end{proposition}

\begin{proof}
As proven in Lemma 13, there exists a unique curve contained in the level
curve $\left\vert f(z)\right\vert =k$ that contains $z_{0}$ as an interior
point. Hence, let $L_{0}$ be the arc-connected component of the level curve $%
\left\vert f(z)\right\vert =k$ contained in $\overline{V}$ that passes
through $z_{0}$ . If $L_{0}$ is neither a closed curve nor crosses the
boundary of $V$, then, as $L_{0}$ is a closed set in the topology of the
complex plane, necessarily $L_{0}$ would have two end points, say $z_{1}$
and $z_{2}$, contained in $V$. Because $z_{1}$, $z_{2}$ are non-critical
points, by applying Lemma 13, there exists a continuation of $L_{0}$, which
is a contradiction because $L_{0}$ is an arc-connected component and
therefore is a maximal set. In this case, it is clear that the continuation
from the points $z_{1}$, $z_{2}$ creates two sets $A$, $B$ (each set $A$ and
$B$ could be reduced to one point) verifying that $A\cap B=\emptyset $. Thus
the result follows.
\end{proof}

The next result is a basic property of the complex numbers.

\begin{lemma}
Let $z_{1}$, $z_{2}$, $z_{3}$ be complex numbers:\newline
a) If $z_{2}=\lambda z_{1}$ with $\lambda >0$ and $\left\vert
z_{2}\right\vert >$ $\left\vert z_{3}\right\vert $, then $\left\vert
z_{1}\right\vert <\left\vert z_{1}+z_{2}+z_{3}\right\vert $.\newline
b) If $z_{2}=-\lambda z_{1}$ with $\lambda >0$ \ and $\left\vert
z_{1}\right\vert >\left\vert z_{2}\right\vert >$ $\left\vert
z_{3}\right\vert $, then $\left\vert z_{1}\right\vert >\left\vert
z_{1}+z_{2}+z_{3}\right\vert $.
\end{lemma}

As we have already seen, a level curve that passes through a non-critical
point of an analytic function is locally a simple curve. Nevertheless, if $%
z_{0}$ is a critical point, then the level curve that passes through $z_{0}$
has, at least, four branches rising from $z_{0}$.

\begin{proposition}
Let $f$ be a non-constant entire function and $z_{0}$ a critical point of $f$
belonging to the level curve $L_{k}:=\left\{ z:\left\vert f(z)\right\vert
=k\right\} $, $k>0$. Let $m$ be the order of the zero of $f^{\prime }(z)$ at
the point $z_{0}$. Then $L_{k}$ has at least $2(m+1)$ branches which meet at
the point $z_{0}$.
\end{proposition}

\begin{proof}
Since $z_{0}$ is a zero of order $m$ of $f^{\prime }(z)$, the Taylor
expansion of $f(z)$ is given by
\begin{equation}
f(z)=f(z_{0})+\sum_{n=m+1}^{\infty }\frac{f^{(n)}(z_{0})}{n!}(z-z_{0})^{n}%
\mbox{, with }f^{(m+1)}(z_{0})\neq 0.  \tag{4.5}
\end{equation}%
Then, there exits $r>0$ so that for every $z$ satisfying $0<\left\vert
z-z_{0}\right\vert \leq r$ we can write
\begin{equation}
f(z)=z_{1}+z_{2}+z_{3},  \tag{4.6}
\end{equation}%
satisfying
\begin{equation}
f^{\prime }(z)\neq 0\mbox{ and }\left\vert z_{1}\right\vert
>\left\vert z_{2}\right\vert >\left\vert z_{3}\right\vert,  \tag{4.7}
\end{equation}%
where $z_{1}$, $z_{2}$, $z_{3}$, (depending on $z$) are given by
\begin{equation*}
z_{1}=f(z_{0}),
\end{equation*}%
\begin{equation*}
z_{2}=\frac{f^{(m+1)}(z_{0})}{(m+1)!}(z-z_{0})^{m+1}
\end{equation*}%
and
\begin{equation*}
z_{3}=\sum_{n=m+2}^{\infty }\frac{f^{(n)}(z_{0})}{n!}(z-z_{0})^{n}.
\end{equation*}%
On the other hand, since for each $\lambda >0$ the equation $z_{2}=\lambda
z_{1}$ has exactly $m+1$ solutions in a certain disk $D(z_{0},r_{\lambda
}):=\left\{ z:\left\vert z-z_{0}\right\vert \leq r_{\lambda }\right\} $,
there exists some $L_{r}>0$ such that the set of solutions of the equation $%
z_{2}=\lambda z_{1}$, when $\lambda \in \left[ 0,L_{r}\right] $, coincides
with the set of all the points of certain $m+1$ radii of the disk $%
D(z_{0},r):=\left\{ z:\left\vert z-z_{0}\right\vert \leq r\right\} $.
Analogously, the set of solutions of the equation $z_{2}=-\lambda z_{1}$
coincides with the set determined by other $m+1$ radii of the disk\ $%
D(z_{0},r)$. Therefore, taking into account (4.6), (4.7) and the preceding
lemma, $D(z_{0},r)$ is partitioned into $2(m+1)$ sectors in such a way that
given one of these sectors, we have that
\begin{equation*}
\left\vert f(z_{0})\right\vert <\left\vert f(z)\right\vert
\end{equation*}%
for all $z$ belonging to one radius, and
\begin{equation*}
\left\vert f(z_{0})\right\vert >\left\vert f(z)\right\vert
\end{equation*}%
for any $z$ of the adjacent radius. Now, by applying Proposition 15 on the
interior of each sector, there exists a branch of the level curve that
passes through $z_{0}$ contained in each sector. The proof is completed and
then the result follows.
\end{proof}

Now we are going to introduce an auxiliary function associated with each $%
G_{n}(s)$ which will be represented by $G_{n}^{\ast }(s)$ and whose level
curves will allow us to give a new characterization of the sets $R_{n}:=%
\overline{\left\{ \mbox{Re}s:G_{n}(s)=0\right\} }$.

\begin{definition}
For every integer $n>2$ we define
\begin{equation*}
G_{n}^{\ast }(s):=G_{n}(s)-p_{k_{n}}^{s},
\end{equation*}%
where $p_{k_{n}}$ is the last prime number such that $p_{k_{n}}\leq n$.
\end{definition}

The characterization of $R_{n}$ in terms of level curves is given in the
following result.

\begin{theorem}
A real number $\sigma \in $ $R_{n}:=\overline{\left\{ \mbox{Re}%
s:G_{n}(s)=0\right\} }$, $n>2$, if and only if the level curve
$\left\vert G_{n}^{\ast }(z)\right\vert =p_{k_{n}}^{\sigma }$
intersects the vertical line $x=$ $\sigma $.
\end{theorem}

\begin{proof}
First of all we will prove the sufficiency. Let $s=$ $\sigma +it$ be a point
of the line $x=$ $\sigma $ belonging to the level curve $\left\vert
G_{n}^{\ast }(z)\right\vert =p_{k_{n}}^{\sigma }$, then
\begin{equation*}
\left\vert G_{n}^{\ast }(\sigma +it)\right\vert =p_{k_{n}}^{\sigma }
\end{equation*}%
and thus there exists some $\theta \in \left[ 0,2\pi \right) $ such that
\begin{equation*}
G_{n}^{\ast }(\sigma +it)=p_{k_{n}}^{\sigma }e^{i\theta }.
\end{equation*}%
That means that
\begin{equation}
1+2^{\sigma +it}+3^{\sigma +it}+...+p_{k_{n}-1}^{\sigma
+it}-p_{k_{n}}^{\sigma }e^{i\theta }+...+n^{\sigma +it}=0.  \tag{4.8}
\end{equation}%
Now, noticing (4.8), for
\begin{equation*}
\mathbf{x=}t(\log 2,\log 3,...,\log p_{k_{n}-1},\theta +\pi ),
\end{equation*}%
the function $F_{n}(\sigma ,\mathbf{x})$, defined by (2.1), satisfies
\begin{equation*}
F_{n}(\sigma ,\mathbf{x})=1+2^{\sigma }e^{it\log 2}+3^{\sigma }e^{it\log
3}+...+
\end{equation*}%
\begin{equation*}
+p_{k_{n}-1}^{\sigma }e^{it\log p_{k_{n}-1}}+p_{k_{n}}^{\sigma }e^{it(\theta
+\pi )}+...+n^{\sigma }e^{it\log n}=
\end{equation*}%
\begin{equation*}
=1+2^{\sigma +it}+3^{\sigma +it}+...+p_{k_{n}-1}^{\sigma
+it}-p_{k_{n}}^{\sigma }e^{i\theta }+...+n^{\sigma +it}=0.
\end{equation*}%
Then, from Theorem 1, it follows that $\sigma \in $ $R_{n}$.

Reciprocally, suppose that $\sigma \in $ $R_{n}$. Then there exists a
sequence of zeros $(s_{m}=\sigma _{m}+it_{m})_{m=1,2,...}$ of $G_{n}(s)$
such that
\begin{equation*}
\sigma =\lim_{m\rightarrow \infty }\sigma _{m}.
\end{equation*}%
Because of the definition of $G_{n}^{\ast }(s)$, at each zero $s_{m}$ of $%
G_{n}(s)$, one has
\begin{equation*}
G_{n}^{\ast }(s_{m})=-p_{k_{n}}^{s_{m}}.
\end{equation*}%
Then, by taking the modulus, we obtain
\begin{equation}
\left\vert 1+2^{\sigma _{m}+it_{m}}+...+p_{k_{n}-1}^{\sigma
_{m}+it_{m}}+...+n^{\sigma _{m}+it_{m}}\right\vert =p_{k_{n}}^{\sigma _{m}}%
\mbox{ for all }m=1,2,...  \tag{4.9}
\end{equation}%
Now, since the sequence $\left( e^{it_{m}}\right) _{m=1,2,...}$is contained
in the unit circle, there exists a subsequence $%
(e^{^{it_{m_{j}}}})_{j=1,2,...}$ such that
\begin{equation*}
\lim_{j\rightarrow \infty }e^{^{it_{m_{j}}}}=e^{i\lambda }\mbox{ for some }%
\lambda \in \left[ 0,2\pi \right).
\end{equation*}%
At thus point, we can rewrite the expression (4.9), for $m_{j}$, $j=1,2,...$%
, and taking the limit $j\rightarrow \infty $, into the expression
\begin{equation*}
\left\vert 1+2^{\sigma }e^{i\lambda \log 2}+3^{\sigma }e^{i\lambda \log
3}+...+p_{k_{n}-1}^{\sigma }e^{i\lambda \log p_{k_{n}-1}}+...+n^{\sigma
}e^{i\lambda \log n}\right\vert =p_{k_{n}}^{\sigma },
\end{equation*}%
which is equivalent to say
\begin{equation*}
\left\vert G_{n}^{\ast }(\sigma +i\lambda )\right\vert =p_{k_{n}}^{\sigma }%
.
\end{equation*}%
That is, the level curve $\left\vert G_{n}^{\ast }(z)\right\vert =$ $%
p_{k_{n}}^{\sigma }$ meets the vertical line $x=$ $\sigma $. Now the proof
is completed.
\end{proof}

As a consequence of the theorem above, we will introduce a real function
with the aim of characterizing the sets $R_{n}$.

\begin{definition}
For $n>2$, we define the real function $A_{n}(x,y)$ as
\begin{equation*}
A_{n}(x,y):=\left\vert G_{n}^{\ast }(x+iy)\right\vert -p_{k_{n}}^{x};%
x,\ y\in
%TCIMACRO{\U{211d} }%
%BeginExpansion
\mathbb{R}
%EndExpansion
.
\end{equation*}
\end{definition}

The characterization of $R_{n}$ by means of $A_{n}(x,y)$ is given in the
following result.

\begin{corollary}
A real number $x$ belongs to $R_{n}$ if and only if $A_{n}(x,y)=0$ for some $%
y\in
%TCIMACRO{\U{211d} }%
%BeginExpansion
\mathbb{R}
%EndExpansion
$. Furthermore
\begin{equation}
A_{n}(x,0)\geq 0\mbox{ for all }x\in \left[ a_{n},b_{n}\right].
\tag{4.10}
\end{equation}
\end{corollary}

\begin{proof}
The first part is a direct consequence of Definition 20 and Theorem19. The
second part immediately follows from Lemma 4 and the inequality in (3.1).
\end{proof}

In order to improve Proposition 5 we will use this real function $A_{n}(x,y)$
for the next result.

\begin{proposition}
Let $n+1$ be a prime number greater than $2$ and
\begin{equation*}
R_{n}:=\overline{\left\{ \mbox{Re}s:G_{n}(s)=0\right\} },\ n=2,3,...%
\end{equation*}%
Then $R_{n}\subset R_{n+1}$.
\end{proposition}

\begin{proof}
As we saw in the proof of Proposition 5, $R_{2}\subset R_{3}$, and thus the
result is valid for $n=2$. Hence let us assume $n>2$. Let $\sigma $ be a
point of $R_{n}$, then there exists a sequence $\left( s_{j}=\sigma
_{j}+it_{j}\right) _{j=1,2,...}$ of zeros of $G_{n}(s)$ such that $\sigma $ $%
=\lim_{j\rightarrow \infty }\sigma _{j}$. Since $n+1$ is a prime number, $%
G_{n+1}^{\ast }(s)=G_{n}(s)$. Then, for every zero $s_{j}$, we have
\begin{equation*}
A_{n+1}(\sigma _{j},t_{j})=-(n+1)^{\sigma _{j}}<0
\end{equation*}%
and, noticing (4.10),
\begin{equation*}
A_{n+1}(\sigma _{j},0)\geq 0.
\end{equation*}%
Then, according to the continuity of $A_{n+1}(x,y)$, there exists $%
t_{j}^{\prime }$ such that
\begin{equation*}
A_{n+1}(\sigma _{j},t_{j}^{\prime })=0.
\end{equation*}%
Now, because of Corollary 21, $\sigma _{j}\in R_{n+1}$ for all $j$, and
then, since $R_{n+1}$ is closed, we get
\begin{equation*}
\lim_{j\rightarrow \infty }\sigma _{j}=\sigma \in R_{n+1}.
\end{equation*}%
This completes the proof.
\end{proof}

Observe that the previous proof, mutatis mutandis, applies to the zeros of $%
G_{n}^{\ast }(s)$ for arbitrary $n>2$ as well. Indeed, the zeros of $%
G_{n}^{\ast }(s)$ supply non-void intervals contained in $R_{n}$, as we will
prove in the next result.

\begin{theorem}
Let $s=\sigma +it$ be a zero of $G_{n}^{\ast }(s)$ with $a_{n}\leq \sigma
<b_{n}$, $n>2$. Then there exists a non-void interval $J$ such that $\sigma
\in $ $J\subset R_{n}$.
\end{theorem}

\begin{proof}
Since $G_{n}^{\ast }(s)=0$, given $\epsilon =$ $p_{k_{n}}^{\sigma }>0$, by
continuity, there exists $r>0$ such that for any $z=x+iy\in \overline{D}(s,r)
$ (the closed disk of center $s$ and radius $r)$ we have
\begin{equation*}
\left\vert G_{n}^{\ast }(x+iy)\right\vert \leq p_{k_{n}}^{\sigma }.
\end{equation*}%
Hence, in particular, by taking $z_{\delta }=\sigma +\delta +it$, with $%
0\leq \delta \leq r$, we get
\begin{equation*}
\left\vert G_{n}^{\ast }(\sigma +\delta +it)\right\vert \leq
p_{k_{n}}^{\sigma }\leq p_{k_{n}}^{\sigma +\delta }.
\end{equation*}%
Then, because of Definition 20, it can be deduced that
\begin{equation*}
A_{n}(\sigma +\delta ,t)\leq 0\mbox{, for any }\delta \in \left[ 0,r\right]
.
\end{equation*}%
Therefore, noticing (4.10) and the continuity of the function $A_{n}(x,y)$,
there exist $t_{\delta }^{\prime }$ such that
\begin{equation*}
A_{n}(\sigma +\delta ,t_{\delta }^{\prime })=0\mbox{, for every }\delta \in %
\left[ 0,r\right].
\end{equation*}%
Consequently, from Corollary 21, it follows that $J:=\left[ \sigma ,\sigma +r%
\right] $ $\subset R_{n}$. The proof is then completed.
\end{proof}

\begin{lemma}
Let $z_{0}$ be a zero of $G_{n}(z)$. Then there exists a disk
\begin{equation*}
U:=\left\{ z:\left\vert z-z_{0}\right\vert <r\right\}
\end{equation*}
such that for all $z\in U$ the formula
\begin{equation}
\frac{\partial \arg G_{n}^{\ast }(z)}{\partial x}=-\frac{1}{\left\vert
G_{n}^{\ast }(z)\right\vert }\frac{\partial \left\vert G_{n}^{\ast
}(z)\right\vert }{\partial y}  \tag{4.11}
\end{equation}%
holds.
\end{lemma}

\begin{proof}
Since $G_{n}(z_{0})=0$, it is true that $G_{n}^{\ast
}(z_{0})=-p_{k_{n}}^{z_{0}}\neq 0$. Then, we can determine a disk $%
U:=\left\{ z:\left\vert z-z_{0}\right\vert <r\right\} $ such that $%
G_{n}^{\ast }(z)\neq 0$ for all $z\in U$. Now, because $U$ is a simply
connected set and noticing $\left[ 1\mbox{, p. 52}\right] $, there exists an
analytic logarithm on $U$ such that%
\begin{equation*}
\log G_{n}^{\ast }(z)=\ln \left\vert G_{n}^{\ast }(z)\right\vert +i\arg
G_{n}^{\ast }(z).
\end{equation*}%
Then the real and imaginary part of $\log G_{n}^{\ast }(z)$, i.e. the two
functions $\ln \left\vert G_{n}^{\ast }(z)\right\vert $ and $\arg
G_{n}^{\ast }(z),$ belong to class $\mathcal{C}^{\infty }$ on $U$ and are
harmonic conjugates, therefore, by the Cauchy-Riemann equations, we obtain
\begin{equation*}
\frac{\partial \arg G_{n}^{\ast }(z)}{\partial x}=-\frac{1}{\left\vert
G_{n}^{\ast }(z)\right\vert }\frac{\partial \left\vert G_{n}^{\ast
}(z)\right\vert }{\partial y}\mbox{ for all }z\in U,
\end{equation*}%
which is the desired formula.
\end{proof}

\begin{lemma}
Let $z_{0}=x_{0}+iy_{0}$ be a zero of $G_{n}(z)$ with $a_{n}<x_{0}<b_{n}$.
If for all $\epsilon >0$ there exist two values $x_{1}\in \left(
x_{0}-\epsilon ,x_{0}\right) $ and $x_{2}\in \left( x_{0},x_{0}+\epsilon
\right) $ of the critical interval $\left[ a_{n},b_{n}\right] $ satisfying $%
A_{n}(x_{1},y)$, $A_{n}(x_{2},y)\neq 0$ for all $y\in
%TCIMACRO{\U{211d} }%
%BeginExpansion
\mathbb{R}
%EndExpansion
$, then $\frac{\partial A_{n}}{\partial x}(x_{0},y_{0})=0$ and consequently
\begin{equation}
\frac{\partial \left\vert G_{n}^{\ast }(z)\right\vert }{\partial x}%
(z_{0})=p_{k_{n}}^{x_{0}}\log p_{k_{n}}.  \tag{4.12}
\end{equation}
\end{lemma}

\begin{proof}
Firstly, as $G_{n}(z)$ does not have any real zero and $G_{n}(\overline{z})=%
\overline{G_{n}(z)}$, we can suppose without loss of generality that $y_{0}>0
$. On the other hand, since $G_{n}(z_{0})=0$, from Definition 20 we have that%
\begin{equation*}
A_{n}(x_{0},y_{0})=0.
\end{equation*}%
By assuming $\frac{\partial A_{n}}{\partial x}(x_{0},y_{0})<0$, it is
deduced the following
\begin{equation*}
\frac{\partial A_{n}}{\partial x}(x_{0},y_{0})=\lim_{x\rightarrow 0^{+}}%
\frac{A_{n}(x_{0}+x,y_{0})}{x}<0.
\end{equation*}%
Therefore, there exists an $\epsilon >0$ such that
\begin{equation*}
A_{n}(x_{0}+x,y_{0})<0\mbox{, for all }x\in \left( 0,\epsilon \right).
\end{equation*}%
On the other hand, because of (4.10),
\begin{equation*}
A_{n}(x_{0}+x,0)\geq 0.
\end{equation*}%
Then, by continuity, there exists some value of $y$, say $y_{x}\in \left(
0,y_{0}\right) $, such that
\begin{equation*}
A_{n}(x_{0}+x,y_{x})=0\mbox{, for every }x\in \left( 0,\epsilon \right).
\end{equation*}%
This contradicts the hypothesis and therefore
\begin{equation*}
\frac{\partial A_{n}}{\partial x}(x_{0},y_{0})\geq 0.
\end{equation*}

Assuming that $\frac{\partial A_{n}}{\partial x}(x_{0},y_{0})>0$, we find
that
\begin{equation*}
\frac{\partial A_{n}}{\partial x}(x_{0},y_{0})=\lim_{x\rightarrow 0^{+}}%
\frac{A_{n}(x_{0}-x,y_{0})}{-x}>0
\end{equation*}%
and conclude that there exists some $\epsilon >0$ such that
\begin{equation*}
A_{n}(x_{0}-x,y_{0})<0\mbox{ for all }x\in \left( 0,\epsilon \right).
\end{equation*}%
Now, by repeating verbatim the above argument, we are led to a contradiction
again. Therefore
\begin{equation}
\frac{\partial A_{n}}{\partial x}(x_{0},y_{0})=0.  \tag{4.13}
\end{equation}%
Finally, looking at (4.13) and according to Definition 20, by taking the
partial derivative of $A_{n}(x,y)$ with respect to $x$ at the point $%
z_{0}=x_{0}+iy_{0}$, the desired formula is derived:
\begin{equation*}
\frac{\partial \left\vert G_{n}^{\ast }(z)\right\vert }{\partial x}%
(z_{0})=p_{k_{n}}^{x_{0}}\log p_{k_{n}}.
\end{equation*}
\end{proof}

The main result of the present paper consists of proving that the real
projections of the simple zeros of $G_{n}(s):=1+2^{s}+...+n^{s}$ are not
isolated points.

\begin{theorem}
Let $s_{0}=\sigma _{0}+it_{0}$ be a simple zero of $G_{n}(s)$, $n>2$. Then,
there exist $\epsilon _{1}$, $\epsilon _{2}\geq 0$, with $\epsilon _{1}+$ $%
\epsilon _{2}>0$, such that $\left( \sigma _{0}-\epsilon _{1},\sigma
_{0}+\epsilon _{2}\right) \subset $ $R_{n}$.
\end{theorem}

\begin{proof}
As we did in the preceding Lemma we can assume, without loss of generality,
that $t_{0}>0$. Since $G_{n}(s_{0})=0$, we have $\left\vert G_{n}^{\ast
}(s_{0})\right\vert =p_{k_{n}}^{\sigma _{0}}$ and so the level curve $%
\left\vert G_{n}^{\ast }(s)\right\vert =p_{k_{n}}^{\sigma _{0}}$ \ passes
through $s_{0}$. Then, let us denote by $L_{0}$ the arc-connected component
that passes through the point $s_{0}.$ Note that $L_{0}$ is not reduced to $%
s_{0}$ because either Lemma 13 (if $s_{0}$ is a non-critical point of $%
G_{n}^{\ast }(s)$), or Proposition 17 (if $s_{0}$ is a critical point of $%
G_{n}^{\ast }(s)$) apply. Now let%
\begin{equation*}
P_{0}:=\left\{ \mbox{Re}s:s\in L_{0}\right\}
\end{equation*}%
\ be its projection on the real axis. Firstly, we claim that $P_{0}$ is a
real interval not reduced to the point $\sigma _{0}$. Indeed, if we suppose
that this is not so, then, necessarily, $L_{0}$ would be contained in the
vertical line of equation $x=$ $\sigma _{0}$. On the other hand, from Lemma
11, the point $\sigma _{0}$ does not belong to $L_{0}.$ Therefore, as $L_{0}$
is a closed set contained in the vertical $x=$ $\sigma _{0}$, there exists a
point $\omega _{0}:=\sigma _{0}+iT\in L_{0}$, where
\begin{equation}
T=\min \left\{ t>0:\sigma _{0}+it\in L_{0}\right\}.  \tag{4.14}
\end{equation}%
Hence, $\omega _{0}$ is necessarily a non-critical point of $G_{n}^{\ast }(s)
$. Indeed, if $\omega _{0}$ were a critical point, then Proposition 17
implies the existence of at least four branches passing through $\omega _{0}$%
, which is impossible because we are assuming that $L_{0}$ is contained in $%
x=$ $\sigma _{0}$. Hence, $\omega _{0}$ is non-critical point and then by
applying Lemma 13, there exists a point $u_{0}:=\sigma _{0}+iT_{0}\in L_{0}$
with $0<T_{0}<T$, which contradicts (4.14). Hence, what we claime is true
and therefore there exist $\delta _{1}$, $\delta _{2}\geq 0$, with $\delta
_{1}+$ $\delta _{2}>0$, such that
\begin{equation}
\left( \sigma _{0}-\delta _{1},\sigma _{0}+\delta _{2}\right) \subset
P_{0}\cap \left[ a_{n},b_{n}\right],  \tag{4.15}
\end{equation}%
where $\left[ a_{n},b_{n}\right] $ is the critical interval of $G_{n}(s)$.

Now, let us consider the two possible cases in virtue of the values of $%
\delta _{1}$ and $\delta _{2}$.

Case 1: $\delta _{2}>0$. In this particular case, (4.15) means that
\begin{equation*}
\left( \sigma _{0},\sigma _{0}+\delta _{2}\right) \subset P_{0}\cap \left[
a_{n},b_{n}\right].
\end{equation*}%
Let $x$ be an arbitrary point of $\left( \sigma _{0},\sigma _{0}+\delta
_{2}\right) $, then $x\in P_{0}$ and therefore there exists a point $%
z=x+iy\in $ $L_{0}$, with $y>0$. Thus $\left\vert G_{n}^{\ast
}(z)\right\vert =p_{k_{n}}^{\sigma _{0}}$ and, in consequence, we obtain%
\begin{equation}
\left\vert G_{n}^{\ast }(z)\right\vert =p_{k_{n}}^{\sigma _{0}}<p_{k_{n}}^{x}%
.  \tag{4.16}
\end{equation}%
Now, according to Definition 20, the expression (4.16) is equivalent to
having
\begin{equation*}
A_{n}(x,y)<0.
\end{equation*}%
On the other hand, from (4.10) we have
\begin{equation*}
A_{n}(x,0)\geq 0
\end{equation*}%
and hence, because of the continuity of $A_{n}(x,y)$, there exists $0\leq
y_{x}\leq y$ such that
\begin{equation*}
A_{n}(x,y_{x})=0.
\end{equation*}%
Now, noticing Corollary 21, it follows that $x\in R_{n}$ and then the Case 1
is proved by just taking $\epsilon _{1}=0$ and $\epsilon _{2}=\delta _{2}$.

Case 2: $\delta _{2}=0$. In this case, (4.14) implies that
\begin{equation*}
\left( \sigma _{0}-\delta _{1},\sigma _{0}\right) \subset P_{0}\cap \left[
a_{n},b_{n}\right]
\end{equation*}%
with $\delta _{1}>0$. Since $\left\vert G_{n}^{\ast }(s_{0})\right\vert \neq
0$, the function $A_{n}(x,y)$ is differentiable on a neighborhood of $%
(\sigma _{0},t_{0})$. Let us assume the existence of some $\epsilon _{1}$ $%
\in \left( 0,\delta _{1}\right] $ such that, for any $x\in \left( \sigma
_{0}-\epsilon _{1},\sigma _{0}\right) $, there exists $y_{x}\in
%TCIMACRO{\U{211d} }%
%BeginExpansion
\mathbb{R}
%EndExpansion
$ such that $A_{n}(x,y_{x})<0$. At this point, by repeating verbatim the
argument of Case 1, the Case 2 follows by taking
\begin{equation*}
\epsilon _{1}=\delta _{1},\ \epsilon _{2}=0.
\end{equation*}

Now, we claim that the Case 2 always leads to the above situation. Indeed,
if this were not so we would have the following: for any $\epsilon _{1}$ $%
\in \left( 0,\delta _{1}\right] $ there exists some $x_{\epsilon }\in \left(
\sigma _{0}-\epsilon _{1},\sigma _{0}\right) $ such that
\begin{equation}
A_{n}(x_{\epsilon },y)\geq 0\mbox{ for all }y\in
%TCIMACRO{\U{211d} }%
%BeginExpansion
\mathbb{R}
%EndExpansion
.  \tag{4.17}
\end{equation}%
Then, under this supposition, by taking $\epsilon _{1}=\frac{1}{m}$, for
sufficiently large $m$, there exists $x_{m}\in \left( \sigma _{0}-\frac{1}{m%
},\sigma _{0}\right) $ such that $A_{n}(x_{m},y)\geq 0$ for all $y\in
%TCIMACRO{\U{211d} }%
%BeginExpansion
\mathbb{R}
%EndExpansion
$. For every fixed value of $y\in
%TCIMACRO{\U{211d} }%
%BeginExpansion
\mathbb{R}
%EndExpansion
$, by taking the limit, it can be found that
\begin{equation}
\lim_{m\rightarrow \infty }A_{n}(x_{m},y)=A_{n}(\sigma _{0},y)\geq 0.
\tag{4.18}
\end{equation}%
Then, since $A_{n}(\sigma _{0},t_{0})=0$ and taking into account (4.18), we
get
\begin{equation*}
\frac{\partial A_{n}}{\partial y}(\sigma _{0},t_{0})=\lim_{y\rightarrow
0^{+}}\frac{A_{n}(\sigma _{0},t_{0}+y)}{y}\geq 0.
\end{equation*}%
On the other hand, (4.18) also implies
\begin{equation*}
\frac{\partial A_{n}}{\partial y}(\sigma _{0},t_{0})=\lim_{y\rightarrow
0^{-}}\frac{A_{n}(\sigma _{0},t_{0}+y)}{y}\leq 0
\end{equation*}%
and therefore
\begin{equation*}
\frac{\partial A_{n}}{\partial y}(\sigma _{0},t_{0})=0.
\end{equation*}%
Consequently, from Definition (20), it follows that
\begin{equation*}
\frac{\partial \left\vert G_{n}^{\ast }(x+iy)\right\vert }{\partial y}%
(\sigma _{0},t_{0})=0
\end{equation*}%
and, according to formula (4.11), we obtain
\begin{equation}
\frac{\partial \arg G_{n}^{\ast }(s_{0})}{\partial x}=0.  \tag{4.19}
\end{equation}%
Finally, using (4.19), formula (4.12), and taking into account that%
\begin{equation*}
\frac{\partial G_{n}(z)}{\partial x}=\frac{\partial G_{n}^{\ast }(z)}{%
\partial x}+\frac{\partial p_{k_{n}}^{z}}{\partial x}=
\end{equation*}%
\begin{equation*}
=e^{i\arg G_{n}^{\ast }(z)}\frac{\partial \left\vert G_{n}^{\ast
}(z)\right\vert }{\partial x}+iG_{n}^{\ast }(z)\frac{\partial \arg
G_{n}^{\ast }(z)}{\partial x}+p_{k_{n}}^{z}\log p_{k_{n}},
\end{equation*}%
at the point $s_{0}$, we get that
\begin{equation*}
G_{n}^{\prime }(s_{0})=\frac{\partial G_{n}(s_{0})}{\partial x}=e^{i\arg
G_{n}^{\ast }(s_{0})}p_{k_{n}}^{\sigma _{0}}\log
p_{k_{n}}+p_{k_{n}}^{s_{0}}\log p_{k_{n}}=
\end{equation*}%
\begin{equation*}
\left( G_{n}^{\ast }(s_{0})+p_{k_{n}}^{s_{0}}\right) \log
p_{k_{n}}=G_{n}(s_{0})\log p_{k_{n}}=0,
\end{equation*}%
which is a contradiction because $s_{0}$ is a simple zero of $G_{n}(s)$.
Therefore the claim is true and the proof of the theorem is completed.
\end{proof}

\section{Existence of simple zeros of $G_{n}(s)$}

In $\left[ 12\mbox{, Proposition }1\right] $ the authors studied the maximum
order of multiplicity of the zeros of the functions $G_{n}(s)$. Furthermore,
there they also showed that all the zeros of $G_{2}(s)$, $G_{3}(s)$ and $%
G_{4}(s)$ are simple. Now, in this section, we present a result that
improves what \ was done in that paper by proving the existence of vertical
strips where all the zeros of $G_{n}(s)$ are simple, provided that $n$ is a
prime number.

\begin{proposition}
Let $n>2$ be a prime number, $b_{n}:=\sup \left\{ \mbox{Re}%
s:G_{n}(s)=0\right\} $ and $b_{n}^{\prime }:=\sup \left\{ \mbox{Re}%
s:G_{n}^{\prime }(s)=0\right\} $, where $G_{n}^{\prime }(s)$
denotes the derivative of $G_{n}(s)$. Then $b_{n}^{\prime
}<b_{n}$.
\end{proposition}

\begin{proof}
Let us define the real functions
\begin{equation*}
f(x):=1+2^{x}+...+(n-1)^{x},
\end{equation*}%
\begin{equation*}
g(x):=\frac{\log 2}{\log n}2^{x}+\frac{\log 3}{\log n}3^{x}...+\frac{\log
(n-1)}{\log n}(n-1)^{x}
\end{equation*}%
and
\begin{equation*}
h(x):=n^{x}.
\end{equation*}%
Firstly let us say that $b_{n}=b_{n,1}$, where $b_{n,1}$ is defined as
\begin{equation*}
b_{n,1}:=\sup \left\{ x\in
%TCIMACRO{\U{211d} }%
%BeginExpansion
\mathbb{R}
%EndExpansion
:f(x)=h(x)\right\}.
\end{equation*}%
Indeed, since $h(x)$, $f(x)$ are positive real functions satisfying
\begin{equation*}
\lim_{x\rightarrow +\infty }\frac{h(x)}{f(x)}=+\infty,
\end{equation*}%
there exists some $x_{0}$ such that $h(x)>f(x)$ for all $x\geq x_{0}$. On
the other hand,
\begin{equation*}
\lim_{x\rightarrow -\infty }\frac{h(x)}{f(x)}=0,
\end{equation*}%
then, from continuity of $\frac{h(x)}{f(x)}$, the number $b_{n,1}$ exists
and is not greater than $x_{0}$. Furthermore, from the definition of $b_{n,1}
$ and the mean value theorem, it follows that
\begin{equation}
h(x)>f(x)\mbox{, for all }x>b_{n,1}.  \tag{5.1}
\end{equation}%
\ Then, given $x>b_{n,1}$, for an arbitrary $y$, we have
\begin{equation*}
\left\vert G_{n}(x+iy)\right\vert =
\end{equation*}%
\begin{equation*}
=\left\vert 1+2^{x}2^{iy}+...+(n-1)^{x}(n-1)^{iy}+n^{x}{}^{+iy}\right\vert
\geq
\end{equation*}%
\begin{equation*}
\geq \left\vert \left\vert n^{x+iy}\right\vert -\left\vert
1+2^{x}2^{iy}+...+(n-1)^{x}(n-1)^{iy}\right\vert \right\vert \geq
\end{equation*}%
\begin{equation*}
\geq \left\vert n^{x}-\left( 1+2^{x}+...+(n-1)^{x}\right) \right\vert >0%
,
\end{equation*}%
which implies that $\left\{ z:\mbox{Re}z>b_{n,1}\right\} $ is a
free-zero region of $G_{n}(s)$ and, in consequence,
\begin{equation}
b_{n}\leq b_{n,1}.  \tag{5.2}
\end{equation}

On the other hand, as $n$ is a prime number, $n=p_{k_{n}}$, so, $f(x)=$ $%
G_{n}^{\ast }(x)$ and $h(x)=p_{k_{n}}^{x}$, for all $x\in
%TCIMACRO{\U{211d} }%
%BeginExpansion
\mathbb{R}
%EndExpansion
$. Now, since $f(b_{n,1})=h(b_{n,1})$, because of Definition 20, we have
\begin{equation*}
A_{n}(b_{n,1},0)=\left\vert G_{n}^{\ast }(b_{n,1})\right\vert
-p_{k_{n}}^{b_{n,1}}=f(b_{n,1})-h(b_{n,1})=0,
\end{equation*}%
which means, from Corollary 21, that $b_{n,1}\in R_{n}\subset \left[
a_{n},b_{n}\right] $ and thus
\begin{equation*}
b_{n,1}\leq b_{n}.
\end{equation*}%
Consequently, from (5.2), we get
\begin{equation}
b_{n}=b_{n,1},  \tag{5.3}
\end{equation}%
as we claimed.

Now by following a similar procedure to what we did above, we obtain
\begin{equation}
b_{n,1}^{\prime }=b_{n}^{\prime }  \tag{5.4}
\end{equation}%
where
\begin{equation*}
b_{n,1}^{\prime }:=\sup \left\{ x\in
%TCIMACRO{\U{211d} }%
%BeginExpansion
\mathbb{R}
%EndExpansion
:g(x)=h(x)\right\}.
\end{equation*}%
Then, according to (5.3) and (5.4), we have
\begin{equation}
f(b_{n})=h(b_{n})\mbox{ and }g(b_{n}^{\prime })=h(b_{n}^{\prime }).
\tag{5.5}
\end{equation}%
Finally, let us say that $b_{n}^{\prime }<$ $b_{n}$. Indeed, since $f(x)>g(x)
$ for all $x\in
%TCIMACRO{\U{211d} }%
%BeginExpansion
\mathbb{R}
%EndExpansion
$, from (5.5), we deduce that $b_{n}^{\prime }\neq $ $b_{n}$. If we assume
that\ $b_{n}^{\prime }>$ $b_{n}$, from (5.3), (5.4) and (5.1), we are led to
\begin{equation*}
h(b_{n}^{\prime })>f(b_{n}^{\prime }).
\end{equation*}%
Now, noticing (5.5), it follows that
\begin{equation*}
g(b_{n}^{\prime })>f(b_{n}^{\prime }),
\end{equation*}%
which is a contradiction because $f(x)>g(x)$ for all $x\in
%TCIMACRO{\U{211d} }%
%BeginExpansion
\mathbb{R}
%EndExpansion
$. Therefore, it is true that%
\begin{equation*}
b_{n}^{\prime }<b_{n}
\end{equation*}%
and the result follows.
\end{proof}

From the previous result, the next corollary can be easily derived.

\begin{corollary}
Let $n>2$ a prime number, $b_{n}:=\sup \left\{
\mbox{Re}s:G_{n}(s)=0\right\} $ and $b_{n}^{\prime }:=\sup \left\{
\mbox{Re}s:G_{n}^{\prime }(s)=0\right\} $. Then, all the zeros of
$G_{n}(s)$ situated in the vertical strip $\left\{ z\in
%TCIMACRO{\U{2102} }%
%BeginExpansion
\mathbb{C}
%EndExpansion
:b_{n}^{\prime }<\mbox{Re}z<b_{n}\right\} $ are simple.
\end{corollary}

\bigskip

\end{document}